\documentclass[a4paper, 10pt]{amsart}

\usepackage[utf8]{inputenc}
\usepackage[T1]{fontenc}

\usepackage{microtype}

\usepackage{enumitem}
\setlist[enumerate,1]{label={\normalfont(\arabic*)}}
\setlist[enumerate,2]{label={\normalfont(\alph*)}}
\setlist[enumerate,3]{label={\normalfont(\roman*)}}

\usepackage{booktabs}

\usepackage{amsmath,amssymb,mathtools}
\usepackage[only,llbracket,rrbracket]{stmaryrd}
\SetSymbolFont{stmry}{bold}{U}{stmry}{m}{n}
\usepackage{bm}

\usepackage{amsthm,thmtools,thm-restate}

\theoremstyle{plain}

\newtheorem{theorem}{Theorem}
\newtheorem{proposition}[theorem]{Proposition}
\newtheorem{corollary}[theorem]{Corollary}
\newtheorem{lemma}[theorem]{Lemma}

\theoremstyle{definition}

\newtheorem{definition}[theorem]{Definition}
\newtheorem{example}[theorem]{Example}

\usepackage[hidelinks]{hyperref}
\usepackage[capitalise,noabbrev]{cleveref}
\Crefname{mtheorem}{Theorem}{Theorems}

\usepackage{orcidlink}
\newcommand\torcidlink[1]{\rlap{\raisebox{-.3ex}{\scalebox{1.4}{\orcidlink{#1}}}}}

\newcommand{\abs}[1]{{\lvert #1 \rvert}}

\newcommand{\pv}[1]{\mathbf{#1}}
\newcommand{\pvI}[1]{\llbracket#1\rrbracket}

\newcommand{\pvM}{\mathbin{\raisebox{.4pt}{\textup{\textcircled{\raisebox{.6pt}{\smaller[2]\textit{m}}}}}}}

\newcommand{\resL}{\mathrm{res}_{\mathrlap{\mathcal{L}}\phantom{\mathcal{R}}}\,}
\newcommand{\resR}{\mathrm{res}_\mathcal{R}\,}
\newcommand{\ResL}{\mathrm{Res}_{\mathrlap{\mathcal{L}}\phantom{\mathcal{R}}}}

\begin{document}

\title[Finite Semigroups Satisfying an Identity $X_1 \dots X_n \approx \rho(X_1, \dots, X_n)$]
  {Finite Semigroups Satisfying an Identity $\bm{X_1 \ldots X_n \approx \rho(X_1, \ldots, X_n)}$}

\author[A. Thumm]{Alexander Thumm \torcidlink{0009-0005-4240-2045}}
\address{University of Siegen, Hölderlinstraße 3, 57076 Siegen, Germany}
\email{alexander.thumm@uni-siegen.de}


\begin{abstract}
  We determine the maximal pseudovarieties of finite semigroups that satisfy an identity of the form $x_1 \dots x_n \approx \rho(x_1, \dots, x_n)$. 
  Applying this classification, we further show that a pseudovariety of permutative semigroups satisfies a common permutation identity if and only if it satisfies an identity of the above form or, equivalently, if it does not contain $\pv{T} = \pvI{x^2 \approx xyx \approx 0}$.
\end{abstract}

\maketitle

\section{Introduction}\label{sec:introduction}

Identities play a fundamental role in the study of semigroups, as they serve to define and characterize classes of semigroups (known as varieties), and provide a basis for analyzing structural properties. 
They are also central to understanding the connections between semigroups, automata, and formal languages.

One may consider the simplest identities to be those that allow the expression of an $n$-ary product of arbitrary elements in an alternative way, i.e., those of the form
\[
 x_1 \dots x_n \approx \rho(x_1, \dots, x_n).
\]

Even in the unary case, these identities give rise to deep and intriguing questions such as the Burnside problem~\cite{Burnside1902} (see also the survey~\cite{Adian2010}), which asks whether a finitely generated group satisfying the identity $x \approx x^{n + 1}$ must necessarily be finite.

In the late 1960s, binary product identities of the form $xy \approx y^{p + 1} x^{q+1}$ ($p + q \geq 1$) were considered by Tully in an unpublished manuscript (see~\cite{Tamura1969}) showing that every semigroup satisfying such an identity is an inflation of a semilattice of Abelian groups with exponents dividing $p$ and $q$.
In particular, every such semigroup is itself commutative.
Tamura~\cite{Tamura1969} considered identities $xy \approx \rho(x, y)$ where $\rho(x,y)$ starts with $y$, ends with $x$, and has length at least three -- every semigroup satisfying such an identity is, once again, an inflation of a semilattice of groups and thus commutative if its subgroups are.
The question of commutativity, reduced from semigroups to groups, has consequently attracted further research; see, e.g., \cite{Kowol1976, Stein2014, Moravec2019, Moravec2020}.

Notably, Tamura’s results have been further generalized in a series of articles by Putcha and Weissglass~\cite{PutchaWeissglass1971,PutchaWeissglass1972}, wherein they explore product identities that may depend on the specific elements involved, as well as identities of higher arity.

\medskip

In this article, we will exclusively concern ourselves with finite semigroups and therefore allow for the right-hand side of a product identity $x_1 \dots x_n \approx \rho(x_1, \dots, x_n)$ to be a nonempty \emph{profinite} word $\rho(x_1, \dots, x_n)$ in the variables $x_1, \dots, x_n$, as is most natural in this setting.
The primary examples of such identities are the members of the following two families.
These identities capture aspects of regularity and commutativity -- properties that also arise for general semigroups, as demonstrated by the above-mentioned results.

\smallskip

$\bullet$\hspace{\labelsep}The first family comprises product identities of the form
\[
     x_1 \dots x_n \approx x_1 \dots x_{i-1} (x_i \dots x_j)^{\omega + 1} x_{j+1} \dots x_n
\]
parametrized by the arity $n$ and the indices $1 \leq i \leq j \leq n$.
These generalize the defining identity $x \approx x^{\omega + 1}$ of completely regular semigroups, which is the only member of arity $n = 1$ of the family.
We call any finite semigroup or pseudovariety satisfying such an identity \emph{almost completely regular}.

\smallskip

$\bullet$\hspace{\labelsep}The second family comprises \emph{permutation identities}, i.e.,
\[
     x_1 \dots x_n \approx x_{1\sigma} \dots x_{n\sigma}
\]
where $\sigma \in \mathfrak{S}_n$ is a nontrivial permutation of the symbols $1, \dots, n$.
Its minimal member is the defining identity $xy \approx yx$ of commutative semigroups.
Semigroups and pseudovarieties satisfying an identity as above are called \emph{permutative}.\footnote{Be aware that some authors define a permutative pseudovariety as a pseudovariety consisting of permutative semigroups, meaning that the permutation identity may differ between its members.}

\medskip

The main result of this article is that the members of the two families above represent the most general of product identities; its formal statement is as follows.

\begin{restatable}{mtheorem}{thmmaximal}\label{thm:main}
  Let $\varepsilon\colon x_1 \dots x_n \approx \rho(x_1, \dots, x_n)$ be a nontrivial $n$-ary product identity.
  Then at least one of the following two statements holds.
  \begin{enumerate}
    \item There exist indices $1 \leq i \leq j \leq n$ such that 
    \[
      \varepsilon \models x_1 \dots x_n \approx x_1 \dots x_{i-1} (x_i \dots x_j)^{\omega + 1} x_{j+1} \dots x_n.
    \]
    \item There exists a nontrivial permutation $\sigma \in \mathfrak{S}_n$ such that 
    \[
      \varepsilon \models x_1 \dots x_n \approx x_{1\sigma} \dots x_{n\sigma}.
    \]
  \end{enumerate}

  In particular, every finite semigroup or pseudovariety satisfying the identity $\varepsilon$ is either almost completely regular, or permutative, or both.
\end{restatable}

We also prove the following characterization, which expresses the deep connection between product identities and the pseudovarieties of nilpotent semigroups 
\[
     \pv{T} \coloneqq \pvI{xyx \approx x^2 \approx 0}
     \quad\text{and}\quad
     \pv{U} \coloneqq \pvI{xy \approx yx, x^2 \approx 0},
\]
as well as their $k$-nilpotent counterparts $\pv{T}_k \coloneqq \pv{T} \cap \pv{N}_k$ and $\pv{U}_k \coloneqq \pv{U} \cap \pv{N}_k$. 

\begin{restatable}{mtheorem}{thmobstructions}\label{thm:obstructions}
  Let $\pv{V}$ be a pseudovariety, and let $k \geq 1$.
  Then the following hold.
  \begin{enumerate}
    \item $\pv{T}_{k+1} \not\subseteq \pv{V}$ if and only if $\pv{V}$ satisfies a nontrivial $k$-ary product identity.
    \item $\pv{U}_{k+1} \not\subseteq \pv{V}$ if and only if $\pv{V}$ satisfies a (regular) $k$-ary expansion identity.\footnote{For the concepts of \emph{regular} product identities and \emph{expansion} identities, see \cref{def:product-identity}.}
  \end{enumerate}

  In particular, it holds that $\pv{T} \not\subseteq \pv{V}$ (or $\pv{U} \not\subseteq \pv{V}$) if and only if $\pv{V}$ satisfies a nontrivial product identity (or regular expansion identity, resp.) of arbitrary arity.
\end{restatable}

Finally, we revisit the aspect of commutativity -- or, more precisely, permutativity. 
By applying the preceding results, we establish the following characterization.

\begin{restatable}{mtheorem}{thmpermutative}\label{thm:permutative}
  Let $\pv{V}$ be a pseudovariety.
  Then $\pv{V}$ is permutative (in the sense that it satisfies a permutation identity) if and only if $\pv{T} \not\subseteq \pv{V}$ and $\pv{V} \subseteq \pv{Perm}$.
\end{restatable}

Here, as usual, $\pv{Perm}$ denotes the pseudovariety of all permutative semigroups, which is \emph{not} a permutative pseudovariety in our sense.

\section{Preliminaries}\label{sec:preliminaries}

The reader is assumed to be familiar with the theory of (finite) semigroups and is referred to the excellent treatments of the subject by Almeida~\cite{Almeida1994}, and Rhodes and Steinberg~\cite{RhodesSteinberg2009} for relevant background material as well as undefined terms.

\medskip

The set of all (nonempty) words over a given alphabet $A$ will be denoted by $A^\ast$ (or $A^{+}$, respectively).
Endowed with the operation of concatenation, $A^+$ is a semigroup.
It is a model of the free $n$-generated semigroup $\Omega_n(\pv{S})$, where $n$ is the cardinality of $A$.
The profinite completion of $A^+$ is the set of all nonempty \emph{profinite words} over $A$.
Given such a profinite word $\rho$, we denote its \emph{length} by $\abs{\rho} \in \mathbb{N} \cup \{\infty\}$ and its \emph{content} (that is, the set of letters appearing in it) by $c(\rho) \subseteq A$.

An \emph{identity}~$\varepsilon$, also called a pseudoidentity in the literature, is a pair of nonempty profinite words~$(\rho_1, \rho_2)$ over a finite alphabet of \emph{variables} $\{x_1, \dots, x_n\}$. 
We usually write the identity $\varepsilon$ as $\varepsilon \colon \rho_1(x_1, \dots, x_n) \approx \rho_2(x_1, \dots, x_n)$.
Such an identity $\varepsilon$ is \emph{finite} if both $\rho_1$ and $\rho_2$ are ordinary words; otherwise $\varepsilon$ is an \emph{infinite} identity.

A finite semigroup $S$ \emph{satisfies} an identity $\varepsilon\colon \rho_1(x_1, \dots, x_n) \approx \rho_2(x_1, \dots, x_n)$, written $S \models \varepsilon$, if the profinite words $\rho_1$ and $\rho_2$ have coincident images under every continuous homomorphism to $S$.
An identity $\varepsilon$ \emph{implies} an identity $\varepsilon'$, written $\varepsilon \models \varepsilon'$, if the implication $S \models \varepsilon \implies S \models \varepsilon'$ holds for all finite semigroups $S$.

\medskip

We use boldface type to denote \emph{pseudovarieties}, i.e., classes of finite semigroups closed under formation of finite direct products, subsemigroups, and homomorphic images.
The pseudovariety of all finite semigroups will be denoted by $\pv{S}$.

A pseudovarity $\pv{V}$ \emph{satisfies} a set of identities $\Delta$, written $\pv{V} \models \Delta$, if every $S \in \pv{V}$ satisfies every $\varepsilon \in \Delta$.
The pseudovariety consisting of all finite semigroups satisfying some set of identities $\Delta$ is denoted by $\pvI{\Delta}$.
To enhance readability we omit brackets when listing identities, and we employ standard abbreviations; that is, we prefer writing, e.g., $\pvI{xy \approx yx, x^2 \approx 0}$ instead of $\pvI{\{ xy \approx yx, ux^2 \approx x^2, x^2u \approx x^2\}}$.

\medskip

As usual, $E(S)$ denotes the set of \emph{idempotents} of a finite semigroup $S$, and we denote the set of its \emph{completely regular} elements by $I(S) \coloneqq \{ s \in S : s = s^{\omega + 1}\}$.

\section{Restricted Identities}\label{sec:restriction}

Our proof of \cref{thm:main} will be by induction on the number of variables appearing in a product identity.
It is based on the following operation on identities.

\begin{definition}\label{def:restriction}
     Let $\varepsilon\colon \rho_1(y_1, \dots, y_n) \approx \rho_2(y_1, \dots, y_n)$ be an identity.
     We obtain the following new identities, which we call the \emph{left} and \emph{right restriction} of $\varepsilon$, respectively, 
     \begin{align*}
       \resL\varepsilon &\colon x \rho_1(y_1, \dots, y_n) \approx x \rho_2(y_1, \dots, y_n),\\
       \resR\varepsilon &\colon \rho_1(y_1, \dots, y_n) z \approx \rho_2(y_1, \dots, y_n) z
     \end{align*}
     where the variables $x$ and $z$ are distinct from $y_1, \dots, y_n$.
\end{definition}

The restricted identities $\resL\varepsilon$ and $\resR\varepsilon$ are similar to the original identity~$\varepsilon$, but they are more general since they only allow for $\varepsilon$ to be applied with a nontrivial left or right context, respectively.\footnote{Naturally, this distinction becomes irrelevant in the context of monoids rather than semigroups.}
For example, restricting $y_1y_2 \approx y_2y_1$ twice on the left and three times on the right yields $x_1x_2y_1y_2z_1z_2z_3 \approx x_1x_2y_2y_1z_1z_2z_3$.

The family of all product identities as well as the two families of identities in \cref{thm:main} are each closed under restrictions (provided that we do not treat the number of variables appearing in them as fixed).
In general, forming left (or right) restrictions gives rise to an operator on the lattice of all pseudovarieties $\mathcal{L}(\pv{S})$.

\begin{proposition}\label{pro:restriction}
  There exists a unique operator ${\ResL} \colon \mathcal{L}(\pv{S}) \to \mathcal{L}(\pv{S})$ such that
  \[
    \ResL \pvI{\Delta} = \pvI{\resL \varepsilon : \varepsilon \in \Delta}
  \]
  for every set of identities $\Delta$.
  Moreover, if $\pv{V'} \subseteq \pv{V}$, then $\ResL\pv{V'} \subseteq \ResL\pv{V}$.
\end{proposition}

As a preparation for the proof of \cref{pro:restriction}, we first deal with the case of finite identities.
Recall that the \emph{deductive closure} of a set $\Delta \subseteq Y^+ \times Y^+$, which we may think of as a set of finite identities over some set of variables $Y$, is the smallest reflexive, symmetric, and transitive set $D_Y(\Delta) \subseteq Y^+ \times Y^+$ which contains $\Delta$, and which is closed under replacement $(\textup{D}_1)$ and substitution $(\textup{D}_2)$, i.e.:
\begin{enumerate}
  \item[($\textup{D}_1$)] If $\rho_1 \approx \rho_2 \in D_Y(\Delta)$, $\sigma, \tau \in Y^\ast$, then $\sigma \rho_1 \tau \approx \sigma \rho_2 \tau \in D_Y(\Delta)$.
 \item[($\textup{D}_2$)] If $\rho_1 \approx \rho_2 \in D_Y(\Delta)$, $y \in Y$, $\sigma \in Y^+$, then $\rho_1[y \gets \sigma] \approx \rho_2[y \gets \sigma] \in D_Y(\Delta)$.
\end{enumerate}
In the above, we write $\rho[y \gets \sigma]$ for the result of substituting $\sigma \in Y^+$ for $y \in Y$ in $\rho$.

\begin{lemma}\label{lem:restriction-closure}
  Let $\Delta \subseteq Y^+ \times Y^+$ be a set of finite identities, $\Sigma = \{ \resL \varepsilon : \varepsilon \in \Delta\}$ its set of left restrictions with respect to a fixed variable $x \not \in Y$, and $X = \{x\} \cup Y$.
  Then, for every $\varepsilon \colon \rho_1 \approx \rho_2 \in D_Y(\Delta)$, it holds that $\resL \varepsilon \colon x\rho_1 \approx x\rho_2 \in D_X(\Sigma)$.
\end{lemma}

\begin{proof}
  We prove the assertion by structural induction on $\varepsilon \in D_Y(\Delta)$.
  If $\varepsilon \in \Delta$, then $\resL \varepsilon \in \Sigma$ and, hence, $\resL\varepsilon \in D_X(\Sigma)$.
  The induction step is trivial if membership of $\varepsilon$ in $D_Y(\Delta)$ is grounded in reflexivity, symmetry, transitivity, or a substitution.

  Finally, suppose that $\varepsilon \colon \sigma\rho_1\tau \approx \sigma\rho_2\tau$ is obtained by replacement, i.e., $\sigma, \tau \in Y^\ast$ and $\rho_1 \approx \rho_2 \in D_Y(\Delta)$ with $x\rho_1 \approx x\rho_2 \in D_X(\Sigma)$.
  Substituting $x \gets x\sigma$ in the latter identity yields $x\sigma \rho_1 \approx x\sigma \rho_2 \in D_X(\Sigma)$ and, hence, $\resL \varepsilon \colon x\sigma \rho_1 \tau \approx x\sigma \rho_2 \tau \in D_X(\Sigma)$ by replacement with an empty prefix and with suffix $\tau$.
\end{proof}

\begin{proof}[Proof of \cref{pro:restriction}]
  The uniqueness part of the assertion is a direct consequence of Reiterman's theorem \cite[Theorem~3.1]{Reiterman1982}, which tells us that every pseudovariety is defined by some set of identities.
  Similarly, the addendum follows easily from Reiterman's theorem.
  To show existence, we adapt a three-step argument devised by Costa~\cite[Proposition~3.1]{Costa2002} and discuss the relevant modifications to it below.

  In each step, one considers a pseudovarity $\pv{V}$ and a basis of identities $\Delta$ thereof, both subject to assumptions.
  Let $\Sigma = \{ \resL \varepsilon : \varepsilon \in \Delta\}$ and $\Sigma' = \{ \resL \varepsilon : \varepsilon \in \Delta'\}$ where $\Delta'$ is the set of all identities satisfied by $\pv{V}$.
  One then proves that $\pvI{\Sigma} = \pvI{\Sigma'}$.
  The inclusion $\pvI{\Sigma} \supseteq \pvI{\Sigma'}$ holds trivially, as $\Delta \subseteq \Delta'$ and, hence, $\Sigma \subseteq \Sigma'$.
  Therefore, it suffices to show that $\pvI{\Sigma} \subseteq \pvI{\Sigma'}$ or, equivalently, that $\Sigma \models \resL \varepsilon$ for every $\varepsilon \in \Delta'$.

  In the first step, we assume that $\pv{V}$ is \emph{locally finite}\footnote{Recall that $\pv{V}$ is locally finite if the free $n$-generated $\pv{V}$-semigroup $\Omega_n(\pv{V})$ is finite for all finite~$n$. 
   Such a pseudovariety is necessarily equational, i.e., it admits a basis of finite identities.} and that $\Delta$ is a basis of \emph{finite} identities.
  If, moreover, $\varepsilon$ is a finite identity (over some common set of variables~$Y$), then the identity $\varepsilon$ is contained in the deductive closure of $D_Y(\Delta)$ by completeness of equational logic \cite[Theorem~10]{Birkhoff1935}.
  Hence, $\resL\varepsilon$ is contained in the deductive closure of $\Sigma$ by \cref{lem:restriction-closure}, and this shows that $\Sigma \models \resL \varepsilon$.
  Using a suitable limit argument (enabled by the assumption that $\pv{V}$ is locally finite) shows that the latter in fact holds for every (not necessarily finite) identity $\varepsilon \in \Delta'$; see~\cite[Proposition~3.1]{Costa2002}.

  In the second step one assumes that $\pv{V}$ is \emph{equational} and that $\Delta$ is a basis of \emph{finite} identities.
  In the third step $\pv{V}$ and $\Delta$ can finally be \emph{arbitrary}.
  Either way, the case at hand reduces to the one handled in the preceding step.
  We omit the details since these reductions are identical to those presented by Costa~\cite[Proposition~3.1]{Costa2002}.
\end{proof}

\section{Permutation and Expansion Identities}\label{sec:auxiliary}

We now proceed to examine product identities, with particular attention to expansion and permutation identities. 
We begin with formal definitions.

\begin{definition}\label{def:product-identity}
  An \emph{$n$-ary product identity} $\varepsilon$ is an identity of the form
  \[
   \varepsilon \colon x_1 \dots x_n \approx \rho(x_1, \dots, x_n)
  \]
  where $x_1, \dotsc, x_n$ are distinct variables.
  We say that $\varepsilon$ is \emph{regular} if each variable~$x_i$ appears in the right-hand side of the identity, i.e., $c(\rho) = \{ x_1, \dots, x_n \}$.
  Moreover, we call the identity $\varepsilon$ an \emph{expansion identity} if its right-hand side has length $\abs{\rho} > n$.
\end{definition}

The right-hand side of a regular $n$-ary product identity $\varepsilon \colon x_1 \dots x_n \approx \rho(x_1, \dots, x_n)$ satisfies $\abs{\rho} \geq n$, and equality holds if and only if $\varepsilon$ is of the form
\[
  \varepsilon \colon x_1 \dotsc x_n \approx x_{1\sigma} \dotsc x_{n\sigma}
\]
for some permutation $\sigma \in \mathfrak{S}_n$ of the symbols $1, \dotsc, n$.
Therefore, every nontrivial regular product identity is either an expansion identity or a permutation identity.

\begin{lemma}\label{lem:regular}
  Let $\varepsilon \colon x_1 \dots x_n \approx \rho(x_1, \dots, x_n)$ be a nontrivial $n$-ary product identity.
  If $\varepsilon$ is not regular, then $\varepsilon \models \varepsilon'$ for some regular $n$-ary expansion identity $\varepsilon'$.
\end{lemma}
\begin{proof}
  If $\varepsilon$ is already regular, then there is nothing to prove.
  Otherwise, choose any variable $x_i$ such that $x_i \not\in c(\rho(x_1, \dots, x_n))$.
  Then $\varepsilon$ implies
  \begin{align*}
    \varepsilon' \colon x_1 \dots x_n 
    &\approx \rho(x_1, \dots, x_{i-1}, x_i, x_{i+1}, \dots, x_n) \\
    &= \rho(x_1, \dots, x_{i-1}, x_i^2, x_{i+1}, \dots, x_n)
    \approx x_1 \dots x_{i-1} x_i^2 x_{i+1} \dots x_n.
    \qedhere
  \end{align*}
\end{proof}

In proving \cref{thm:main}, the preceding lemma permits us to restrict attention to regular expansion identities. 
The following elementary observation allows us to further confine the analysis to those that are infinite. 
While the remainder of the argument will be carried out pointwise, which ultimately returns us to the setting of finite identities, the initial reduction to infinite identities enables us to assume that the lengths of the corresponding right-hand sides are arbitrarily large.

\begin{lemma}\label{lem:infinite}
  Let $\varepsilon \colon x_1 \dots x_n \approx \rho(x_1, \dots, x_n)$ be a (regular) expansion identity.
  Then there exists an infinite (regular) $n$-ary expansion identity $\varepsilon'$ with $\varepsilon \models \varepsilon'$.
\end{lemma}
\begin{proof}
  If $\varepsilon$ is already infinite, then there is nothing to prove.
  Otherwise, consider the sequence of words $\rho_0, \rho_1, \dots$ where $\rho_0 = x_1 \dots x_n$ and $\rho_{i+1}$ is obtained from $\rho_i$ by replacing its prefix $y_1 \dots y_n$ of length $n$ by $\rho(y_1, \dots, y_n)$.
  If $\varepsilon$ is regular, then we have $c(\rho_i) = \{x_1, \dots, x_n\}$ for all $i \geq 0$.
  Moreover, as $\abs{\rho} = n + k$ for some $k > 0$, we have that $\abs{\rho_i} = n + ik \to \infty$ for $i \to \infty$.
  Furthermore, $\varepsilon \models \rho_0 \approx \rho_i$ for all $i \geq 0$.

  By sequential compactness of the space of profinite words over $\{x_1, \dots, x_n\}$, the sequence $\rho_0, \rho_1, \dots$ has an accumulation point $\rho'$. 
  Since content and length are continuous, and the set of identities implied by any fixed set of identities is closed, the identity $\varepsilon' \colon x_1 \dots x_n \approx \rho'(x_1, \dots, x_n)$ has the desired properties.
\end{proof}

The following proposition summarizes the results of the above inquiry.

\begin{proposition}\label{pro:exp-or-per}
  Let $\varepsilon \colon x_1 \dots x_n \approx \rho(x_1, \dots, x_n)$ be a nontrivial product identity.
  Then precisely one of the following alternatives holds.
  \begin{enumerate}
    \item There exists an infinite regular $n$-ary expansion identity $\varepsilon'$ with $\varepsilon \models \varepsilon'$.
    \item The identity $\varepsilon$ is a permutation identity.
  \end{enumerate}
\end{proposition}

\begin{proof}
  If $\varepsilon$ is not a permutation identity, then it implies a regular $n$-ary expanions identity $\varepsilon'$ by \cref{lem:regular}, and the latter can be assumed infinite by \cref{lem:infinite}.

  To see that the alternatives are indeed mutually exclusive, simply note that the monoigenic semigroup $C_{n+1,1} = \langle a : a^{n+1} = a^{n+2} \rangle$ satisfies every permutation identity, but does not satisfy a $k$-ary expansion identity for any $k \leq n$.
\end{proof}

The remainder of this section examines how satisfying an expansion identity affects the structure of a finite semigroup, focusing in particular on its idempotent and completely regular elements.

\begin{lemma}\label{lem:nil-extension}
  Let $\varepsilon \colon x_1 \dots x_n \approx \rho(x_1, \dots, x_n)$ be an expansion identity.
  Furthermore, let $S$ is a finite semigroup with $S \models \varepsilon$, and let $E = E(S)$.
  Then $S^n = SES$.
\end{lemma}
\begin{proof}
  We may assume that $\varepsilon$ is finite and, by \cref{lem:infinite}, that $\abs{\rho} \geq m$ where $m = \abs{S}$.
  Then the inclusions $S^n \subseteq S^m \subseteq SES$ hold, the first being due to $\varepsilon$, and the second being a standard application of the pigeonhole principle \cite[Proposition~3.7.1]{Almeida1994}.

  The reverse inclusion $SES \subseteq S^n$ trivially holds in every semigroup.
\end{proof}

For the expansion identities of the following type, we can furthermore relate the set of idempotents $E(S)$ with the set of completely regular elements $I(S)$.

\begin{definition}
  Let $\varepsilon \colon x_1 \dots x_n \approx \rho(x_1, \dots, x_n)$. 
  We call $\varepsilon$ \emph{left-} or \emph{right-primitive} if $\rho(x_1, \dots, x_n)$ cannot be decomposed as $x_1 \rho'(x_2, \dots, x_n)$ or $\rho'(x_1, \dots, x_{n-1}) x_n$, respectively.
  The identity $\varepsilon$ is called \emph{primitive} if it is left- and right-primitive.
\end{definition}

Note that an $n$-ary product identity is thus left- or right-primitive if, and only if, it is \emph{not} a left or right restriction, respectively, of some $(n-1)$-ary product identity.

\begin{lemma}\label{lem:ideal}
  Let $\varepsilon \colon x_1 \dots x_n \approx \rho(x_1, \dots, x_n)$ be an expansion identity.
  Furthermore, let $S$ be a finite semigroup with $S \models \varepsilon$, and let $E = E(S)$ and $I = I(S)$.
  \begin{enumerate}
    \item\label{lem:ideal-itm1} If $\varepsilon$ is left-primitive, then $SE = I$ and $SI = I$.
    \item\label{lem:ideal-itm2} If $\varepsilon$ is primitive, then $S^n = SES = I$.
  \end{enumerate}
\end{lemma}

Putcha and Weissglass previously established a variant of this result under an assumption slightly less general than our notion of primitivity; see \cite[Lemma~1.6]{PutchaWeissglass1972}.

\begin{proof}
  \textit{Ad \ref{lem:ideal-itm1}.}
  As above, we may assume that $\varepsilon$ is finite.\footnote{Note that specifying the first and last letter of profinite words, and which letters are to appear zero, exactly once, or multiple times in them defines a closed set in the profinite topology.}
  Let $s \in S$ and $e \in E$.
  Moreover, let $x_i$ be the first letter of $\rho$, i.e., $\rho(x_1, \dots, x_n) = x_i \rho'(x_1, \dots, x_n)$.

  If $i = 1$, then the variable $x_1$ also appears in $\rho'$ by our assumption on $\rho$.
  Therefore, it appears $k+1 \geq 2$ times in $\rho$.
  It follows that $se \in I$, since
  \[
    se = \rho(se, e, \dots, e) = (se)^{k+1}.
  \] 

  If $i \neq 1$, then $se = \rho(se, e, \dots, e) = e \rho'(es, e, \dots, e) = ese$.
  Then $es \in I$, since substituting $ese$ for a variable $x_j$ appearing $k + 1 \geq 2$ times in $\rho$ yields 
  \[
    se = ese = \rho(e, \dots, e, ese, e, \dots, e) = (ese)^{k+1} = (se)^{k+1}.
  \]

  The above argument shows that $SE \subseteq I$.
  The equalities $SE = I$ and $SI = I$ then follow from the general fact that $I \subseteq SE$; indeed, if $s \in I$, then $s = ss^\omega \in SE$.

  \medskip

  \textit{Ad \ref{lem:ideal-itm2}.}
  If $\varepsilon$ is primitive, then $SES = I$ by the previous item and its left-right dual.
  The remaining equality $S^n = SES$ follows from \cref{lem:nil-extension}.
\end{proof}

The restriction of identities, along with the closely related concept of primitivity, appears to play a central role in addressing related problems.
For example, in the context of determining interrelations between permutation identities, primitivity also appears implicitly in the work of Putcha and Yaqub \cite[Theorem~2]{PutchaYaqub1971}.
Almeida~\cite{Almeida1986a, Almeida1986b, Almeida1985} investigated semigroups satisfying a nontrivial (left- or right-) primitive permutation identity, calling them strongly (left or right) permutative semigroups.

\section{Proof of Theorem~\ref*{thm:main}}\label{sec:proof}

Having gathered the necessary tools, we now proceed to prove the main theorem, which is restated below for the reader’s convenience.

\thmmaximal*

\begin{proof}
  In view of \cref{pro:exp-or-per} we can assume that $\varepsilon$ is a regular expansion identity.
  We claim that such an identity necessarily implies an identity of the form
  \[
    \varepsilon \models \delta \colon x_1 \dots x_n \approx x_1 \dots x_{i-1} (x_i \dots x_j)^{\omega + 1} x_{j+1} \dots x_n
  \]
  with $1 \leq i \leq j \leq n$.
  If $\varepsilon$ is primitive, which is always the case when $n = 1$, then this claim follows from item \ref{lem:ideal-itm2} of \cref{lem:ideal}. 
  Indeed, the lemma asserts that $S^n \leq S$ is completely regular if $S \models \varepsilon$; that is, $S \models x_1 \dots x_n \approx (x_1 \dots x_n)^{\omega + 1}$.

  Suppose, without loss of generality, that $\varepsilon$ is not left-primitive.
  Then $\varepsilon = \resL \varepsilon'$ for some regular expansion identity $\varepsilon'$ of arity $n-1$.
  By induction on the arity $n$, 
  \[
    \varepsilon' \models \delta' \colon x_2 \dots x_n \approx x_2 \dots x_{i-1} (x_i \dots x_j)^{\omega + 1} x_{j+1} \dots x_n
  \]
  for some $2 \leq i \leq j \leq n$ or, equivalently, $\pvI{\varepsilon'} \subseteq \pvI{\delta'}$.
  By \cref{pro:restriction}, we can apply the operator ${\ResL}$ to this inclusion of pseudovarieties and obtain
  \[
    \pvI{\varepsilon} = \ResL \pvI{\varepsilon'} \subseteq \ResL \pvI{\delta'} = \pvI{\delta}. \qedhere
  \]
\end{proof}

To conclude this section, we emphasize an important feature of the foregoing proof, namely its constructive nature.
Specifically, suppose that a nontrivial product identity $\varepsilon\colon x_1 \dots x_n \approx \rho(x_1, \dots, x_n)$ is presented in an effective manner; for instance, with $\rho = \rho(x_1, \dots, x_n)$ given as a word or as an $\omega$-term.
More precisely, it suffices that one can extract the prefix and suffix of $\rho$ of length $n$, as well as the sets of variables that occur zero, exactly once, or multiple times within $\rho$.
Under these conditions, the above proof effectively produces an identity -- among those specified in \cref{thm:main} -- that is implied by the given product identity $\varepsilon$.

Be aware, however, that the given identity $\varepsilon$ may also imply other such identities, and it may well be of interest to determine all of them.
The following example illustrates this. 
The identity it discusses closely resembles -- and is implied by -- the finite identities mentioned in the introduction that were studied by Tully.

\begin{example}
  Let $\varepsilon \colon xy \approx y^{\omega + 1} x^{\omega + 1}$.
  For this identity, the proof of \cref{thm:main} leads directly to item $(2)$ of \cref{lem:ideal}; hence $\varepsilon \models xy \approx (xy)^{\omega + 1}$.
  In other words, given a finite semigroup $S$ satisfying $\varepsilon$, the ideal $S^2 \leq S$ is completely regular.

  On the other hand, a straight-forward computation shows that 
  \[
   \varepsilon \models xy \approx y^{\omega + 1} x^{\omega + 1} \approx (x^{\omega + 1})^{\omega + 1}(y^{\omega + 1})^{\omega + 1} = x^{\omega + 1} y^{\omega + 1} \approx yx.
  \]
  Since, in particular, $\varepsilon \models xy \approx x^{\omega + 1} y^{\omega + 1}$, the map $\psi \colon S \to S^2$ given by $s\psi = s^{\omega+1}$ witnesses $S$ to be an inflation of its completely regular ideal $S^2$.
  The latter is a semilattice of Abelian groups by commutativity.
  Using the same identity we obtain that $\varepsilon \models xy \approx x^{\omega+1}y^{\omega+1} = x^\omega x^{\omega+1}y^{\omega+1} \approx x^{\omega+1}y$ and, similarly, $\varepsilon \models xy \approx xy^{\omega + 1}$.
\end{example}

\section{Obstructions}\label{sec:obstructions}

In this section, we examine specific pseudovarieties of nilpotent semigroups that serve as primary obstructions to the satisfaction of product and expansion identities.

Recall that a semigroup $S$ with zero is called \emph{$k$-nilpotent} if every product of length at least $k$ evaluates to zero in $S$.
The pseudovariety $\pv{N}_k = \pvI{x_1 \dots x_k \approx 0}$ consists of finite such semigroups, and $\pv{N} = \bigcup_k \pv{N}_k = \pvI{x^\omega \approx 0}$ is the pseudovariety of all finite \emph{nilpotent} semigroups.
Of particular interest to us are its subpseudovarieties 
\begin{equation*}
  \pv{T} \coloneqq \pvI{xyx \approx x^2 \approx 0}
  \quad\text{and}\quad
  \pv{U} \coloneqq \pvI{xy \approx yx, x^2 \approx 0},
\end{equation*}
as well as $\pv{T}_k \coloneqq \pv{T} \cap \pv{N}_k$ and $\pv{U}_k \coloneqq \pv{U} \cap \pv{N}_k$, due to the following result.

{\renewcommand\footnote[1]{}\thmobstructions*}

Clearly, a nilpotent semigroup satisfies a $k$-ary expansion identity if and only if it is $k$-nilpotent (see, e.g., \cref{lem:nil-extension}).
Therefore, \cref{thm:obstructions} entails the following result by Almeida and Reilly~\cite[Proposition~4.4]{AlmeidaReilly1984}.

\begin{corollary}[Almeida, Reilly]
  Suppose that $\pv{V}$ is a pseudovariety with $\pv{V} \subseteq \pv{N}$.
  Then, for every $k \geq 1$, it holds that either $\pv{U}_{k+1} \subseteq \pv{V}$ or $\pv{V} \subseteq \pv{N}_k$.
\end{corollary}

Before we tend to the proof of \cref{thm:obstructions}, let us briefly describe a particular family of models $T_k$ of the free $k$-generated $\pv{T}$-semigroups $\Omega_k(\pv{T})$.
For each $k \geq 1$, the semigroup $T_k$ consists of all nonempty \emph{injective words}\footnote{A word is called injective if each of its letters occurs at most once.} over $A_k = \{ a_1, \dots, a_k \}$ and an additional element $0 \in T_k$, which acts as a zero element.
The composition of injective words equals their concatenation provided that the latter is an injective word; the composition is equal to $0$ otherwise.
In particular, $a_1, \dots, a_k \in T_k$, viewed as injective words of length one, jointly generate the semigroup $T_k$, and every nonzero element of $T_k$ is uniquely expressible as a product of these generators.

A word over $A_k$ of length strictly exceeding $k$ cannot be injective; hence, the semigroup $T_k$ is $(k+1)$-nilpotent and, in particular, finite.
Since any word of the form $xyx$ or $x^2$ with $x,y \in A_k^+$ is certainly not injective, we see that $T_k \in \pv{T}_{k+1}$.

To see that $T_k$ coincides with the free $k$-generated $\pv{T}$-semigroup, i.e., $T_k \cong \Omega_k(\pv{T})$, it suffices to observe that a word $w \in A_k^+$ is injective if and only if it does not contain a factor of the form $xyx$ or $x^2$ with $x, y \in A_k^+$; therefore, $T_k$ is the Rees quotient of $A_k^+ \cong \Omega_k(\pv{S})$ by the ideal generated by all words $xyx$ and $x^2$ with $x,y \in A_k^+$.

Since $\pv{U} = \pv{T} \cap \pv{Com}$, the free $k$-generated $\pv{U}$-semigroup $\Omega_k(\pv{U})$ is the maximal commutative quotient of $\Omega_k(\pv{T})$.
In particular, $\Omega_k(\pv{U}) \in \pv{U}_{k+1}$.

\begin{proof}[Proof of \cref{thm:obstructions}]
  As is evident from the above, the semigroup $\Omega_k(\pv{T}) \in \pv{T}_{k+1}$ does not satisfy any permutation identity on $k$ variables.
  Indeed, the map sending $\sigma \in \mathfrak{S}_k$ to the injective word $a_{1\sigma} \dots a_{k\sigma} \in T_k \cong \Omega_k(\pv{T})$ is clearly one-to-one.
  On the other hand, the semigroup $\Omega_k(\pv{U}) \in \pv{U}_{k+1} \subseteq \pv{T}_{k+1}$ does not satisfy a $k$-ary expansion identity since $\Omega_k(\pv{U})$ is nilpotent but clearly not $k$-nilpotent.
  Together with \cref{lem:regular} this proves sufficiency of the conditions in $(1)$ and $(2)$.

  For necessity, assume that $\pv{T}_{k+1} \not\subseteq \pv{V}$ (or $\pv{U}_{k+1} \not\subseteq \pv{V}$).
  Then $\pv{V}$ satisfies some identity $\varepsilon \colon \rho_1(x_1, \dots, x_m) \approx \rho_2(x_1, \dots, x_m)$ which does not hold in $\pv{T}_{k+1}$ (or $\pv{U}_{k+1}$, respectively) by Reiterman's theorem \cite[Theorem~3.1]{Reiterman1982}.
  At least one of the two sides of the identity $\varepsilon$ is a finite injective word of length $n \leq k$ for otherwise both sides would evaluate to zero in every semigroup in $\pv{T}_{k+1}$ (which is impossible since we assumed that $\pv{T}_{k+1} \not\models \varepsilon$).
  Upon renaming variables, we can write the identity in question as $\varepsilon \colon x_1 \dots x_n \approx \rho(x_1, \dots, x_m)$ where $1 \leq n \leq m$.
  Moreover, we may assume that $x_1, \dots, x_n \in c(\rho(x_1, \dots, x_m))$:  if $x_i \not \in c(\rho(x_1, \dots, x_m))$ with $1 \leq i \leq n$, then we can deduce, as in the proof of \cref{lem:regular}, that the identity $\varepsilon$ implies the regular expansion identity $\varepsilon' \colon x_1 \dots x_n \approx x_1 \dots x_{i-1}  x_i^2  x_{i+1} \dots x_n$. 

  If $\rho(x_1, \dots, x_m)$ is a finite word of length $n$, then $\rho(x_1, \dots, x_m) \approx x_{1\sigma} \dots x_{n\sigma}$ for some nontrivial permutation $\sigma \in \mathfrak{S}_n$;
  in particular, $\varepsilon$ is a nontrivial product identity.
  Clearly, this case cannot occur when $\pv{U}_{k+1} \not\models \varepsilon$ since $\pv{U}_{k+1} \subseteq \pv{Com}$.

  Otherwise, the profinite word $\rho(x_1, \dots, x_m)$ has length strictly exceeding $n$; hence, so does the right-hand side of the identity $\varepsilon' \colon x_1 \dots x_n \approx \rho(x_1, \dots, x_n, x_1, \dots, x_1)$.
  The latter is a regular expansion identity.
  Since $\varepsilon \models \varepsilon'$, this concludes the proof.
\end{proof}

\section{Independence}\label{sec:independence}

\Cref{thm:main} presents a collection of $n$-ary product identities of maximal generality, prompting the natural question: \emph{what are the implications among them?}

\medskip

For a nontrivial permutation $\sigma \in \mathfrak{S}_n$ acting on the set $\{1, \dots, n\}$, we denote the pseudovariety defined by the corresponding permutation identity by
\[
  \pv{P}^n_\sigma \coloneqq \pvI{x_1 \dots x_n \approx x_{1\sigma} \dots x_{n\sigma}}.
\]

The maximal pseudovarieties among the $\pv{P}^n_\sigma$ with $\sigma \in \mathfrak{S}_n$ and $\sigma \neq \mathrm{id}$ are in bijective correspondence with the minimal nontrivial subgroups of the symmetric group $\mathfrak{S}_n$, as is evident from the following observation.

\begin{lemma}
  Let $\sigma_1, \sigma_2 \in \mathfrak{S}_n$.
  Then $\pv{P}^n_{\sigma_1} \subseteq \pv{P}^n_{\sigma_2}$ if and only if $\langle \sigma_1 \rangle_{\mathfrak{S}_n} \supseteq \langle \sigma_2 \rangle_{\mathfrak{S}_n}$.
\end{lemma}
\begin{proof}
  The sufficiency of the condition is obvious.
  To see necessity, consider the maximal quotient belonging to $\pv{P}^n_{\sigma_2}$ of the semigroup $T_n$ defined in \cref{sec:obstructions}.
\end{proof}

For a discussion of inclusions where the parameter $n$ is also allowed to vary between pseudovarieties, the reader is referred to the work of Putcha and Yaqub~\cite{PutchaYaqub1971}.

\medskip

We now turn to the family of almost completely regular pseudovarieties that are defined for $1 \leq i \leq j \leq n$ by 
\[
  \pv{ACR}_{i,j}^{n} \coloneqq \pvI{x_1 \dots x_n \approx x_1 \dots x_{i-1} (x_i \dots x_j)^{\omega + 1} x_{j+1} \dots x_n}.
\]

The following argument shows that they are incomparable with the former.

\begin{lemma}
  Let $1 \leq i \leq j \leq n$ and $\sigma \in \mathfrak{S}_n \setminus \{\mathrm{id}\}$.
  Then $\smash{\pv{ACR}_{i,j}^{n}} \not\subseteq \pv{P}^n_\sigma \not\subseteq \smash{\pv{ACR}_{i,j}^{n}}$.
\end{lemma}
\begin{proof}
  The pseudovariety $\smash{\pv{ACR}_{i,j}^{n}}$ contains all finite groups.
  Since every group in $\pv{P}^n_\sigma$ is clearly commutative, we have $\smash{\pv{ACR}_{i,j}^{n}} \not\subseteq \pv{P}_\sigma$.
  Conversely, the pseudovariety~$\pv{P}^n_\sigma$ contains all finite monogenic semigroups, whereas $C_{r,1} = \langle a : a^r = a^{r+1} \rangle$ is contained in the pseudovariety~$\smash{\pv{ACR}^{n}_{i,j}}$ if and only if $r \leq n$. 
  Therefore, $\pv{P}^n_\sigma \not\subseteq \smash{\pv{ACR}_{i,j}^{n}}$.
\end{proof}

Since every pseudovariety $\pv{V}$ satisfies $\pv{V} \subseteq \ResL \pv{V}$ and $\pv{V} \subseteq \pv{V} \pvM \pv{N}_k$ for all $k$, there are many inclusions among the pseudovarieties $\smash{\pv{ACR}_{i,j}^{n}}$.
Such inclusions are, however, precluded under a fixed value of the parameter $n$.

\begin{lemma}\label{pro:max-CR}
  Let $n \geq 1$.
  The $\smash{\pv{ACR}_{i,j}^{n}}$ with $1 \leq i \leq j \leq n$ are pairwise incomparable.
\end{lemma}
\begin{proof}
  To see the incomparability we consider the following semigroups.

  \smallskip

  $\bullet$\hspace{\labelsep}For $1 \leq k$, let $W_k$ be the semigroup consisting of all nonempty words of length at most $k$ over the binary alphabet $\{a,b\}$.
  The composition of such words is the longest prefix of length at most $k$ of their concatenation.
  This semigroup is clearly finite and belongs to the pseudovariety $\pv{K}_k = \pvI{x_1 \dots x_k \approx x_1 \dots x_k  y}$.

  Since the prefixes of length $j$ of the left- and right-hand side of the defining identity of $\smash{\pv{ACR}^{n}_{i,j}}$ coincide, it holds that $W_k \in \pv{K}_k \subseteq \smash{\pv{ACR}^{n}_{i,j}}$ whenever $1 \leq k \leq j$.
  Conversely, if $1 \leq j < k$, then $W_{k} \not\in \smash{\pv{ACR}^{n}_{i,j}}$ as is witnessed by the computation
  \[
    a^{j}b^{k-j} = a^{j}b^\omega = a^{i-1}a^{j-i + 1}b^\omega \neq a^{i-1}(a^{j-i+1})^{\omega + 1}b^\omega = a^{\omega} = a^k.
  \]

  $\bullet$\hspace{\labelsep}For $1 \leq k \leq n$, let $V_{k,n}$ denote the semigroup with zero given by
  \[
    V_{k,n} \coloneqq \langle a, b : a^n = a^{n+1}, ab = b, b a^{n - k + 1} = 0 \rangle.
  \]
  Its elements are $a = a^1, \dots, a^n = a^\omega$, $b = ba^{0}, \dots, ba^{n-k}$, and $ba^{n-k+1} = b^2 = 0$.

  Let $s_1 \dots s_n$ be an $n$-ary product in $V_{k,n}$.
  We will first consider the case where, for some $1 \leq l \leq n$, the factor $s_l$ is of the form $ba^{r_l}$ with $r_l \geq 0$, all remaining factors are of the form $s_m = a^{r_m}$ with $r_m \geq 1$, and where $r \coloneqq r_l + r_{l+1} + \dots + r_n \leq n - k$.

  Such a product then satisfies $s_1 \dots s_n = ba^r \neq 0$.
  Furthermore, for $1 \leq i \leq j \leq n$, the product $s_1 \dots s_{i-1} (s_i \dots s_j)^{\omega+1} s_{j+1} \dots s_n$ equals $ba^r$ if $j < l$ and $0$ otherwise.
  In particular, this shows that $V_{k,n} \not\in \smash{\pv{ACR}^{n}_{i,j}}$ whenever $1 \leq k \leq j$ (by setting $l = k$, $r_l = 0$, and $r_1 = \dots = r_{l-1} = r_{l+1} = \dots = r_n = 1$ in the above).
  On the other hand, the product $s_1 \dots s_n$ obeys the defining identity of $\smash{\pv{ACR}^{n}_{i,j}}$ whenever $1 \leq j < k$, since we always have $k \leq l$ in the above as $n - l \leq r_l + r_{l+1} + \dots + r_n = r \leq n - k$.

  Every $n$-ary product $s_1 \dots s_n$ in $V_{k,n}$ not of the form discussed above satisfies the equality $s_1 \dots s_n = s_1 \dots s_{i-1} (s_i \dots s_j)^{\omega + 1} s_{j+1} \dots s_n$ for all $1 \leq i \leq j \leq n$. 
  Indeed, for such a product either all of the factors $s_1, \dots, s_n$ are powers of $a$, in which case both sides evaluate to $a^\omega$, or both sides of the identity evaluate to $0$.
  Since the latter is easy to verify in a case-by-case manner, we omit the details.

  \smallskip

  In summary: for $1 \leq k,j \leq n$, we have $W_k \in \pv{ACR}^{n}_{i,j}$ if and only if $V_{k,n} \not\in \pv{ACR}^{n}_{i,j}$ if and only if $k \leq j$. 
  Therefore, $\smash{\pv{ACR}^{n}_{i',j'}} \subseteq \smash{\pv{ACR}^{n}_{i,j}}$ implies $j' = j$ and, by symmetry, also $i' = i$; the pseudovarities are therefore pairwise incomparable.
\end{proof}

Fixing the parameter $n$, the above discussion provides complete information about inclusions among the pseudovarieties $\pv{P}^n_\sigma$ and $\pv{ACR}^n_{i,j}$ under consideration. 
However, a deeper understanding of the implications among $n$-ary product identities requires examining inclusions involving intersections of these pseudovarieties, i.e., to determine the meet subsemilattice of $\mathcal{L}(\pv{S})$ that they generate.
Both this issue and its generalization to varying $n$ constitute open problems for future research.

\section{Permutative Pseudovarieties}\label{sec:permutative}

Recall from the introduction that we call a semigroup or pseudovariety \emph{permutative} if it satisfies a permutation identity.
The finite permutative semigroups form a pseudovariety, which we will denote by $\pv{Perm} = \pvI{x^\omega  y_1y_2  z^\omega \approx x^\omega  y_2y_1  z^\omega}$.\footnote{For the defining identity of the pseudovariety $\pv{Perm}$, see \cite[Corollary~3.10]{Almeida1986a}.}

This pseudovariety is \emph{not} permutative, as there is no permutation identity satisfied by all its members simultaneously.
In fact, since $\pv{T} \subseteq \pv{Perm}$, this also holds for nontrivial product identities by \cref{thm:obstructions}.
The following result, which we prove in this section, shows this to be the only obstruction.

\thmpermutative*

The pseudovarieties of permtuative semigroups have been characterized by Almeida in terms of forbidden members \cite{Almeida1986a, Almeida1986b}.\footnote{The details of this characterization are also discussed in Almeida's book~\cite[Section~6.5]{Almeida1994}.}
Naturally, non-Abelian groups must be excluded. 
The minimal such groups were determined by Rédei~\cite[Satz~8]{Redei1947}.

Moreover, it is also necessary to exclude other minimal noncommutative monoids, viz.\ $B^1(2,1)$, $B^1(1,2)$, and $N^1$, as determined by Margolis and Pin~\cite[Theorem~2.1]{MargolisPin1984}.
Notably, the monoids $B^1(2,1)$ and $B^1(1,2)$ generate the pseudovarieties
\[
  \pv{LRB} \coloneqq \pvI{x^2 \approx x, xyx \approx xy}
  \quad\text{and}\quad
  \pv{RRB} \coloneqq \pvI{x^2 \approx x, xyx \approx yx}
\]
of \emph{left-} and \emph{right-regular} bands, respectively, which are precisely the minimal nonpermutative pseudovarieties of bands (see \cite[Theorem~10]{YamadaKimura1958} and \cite{Biryukov1970,Fennemore1971,Gerhard1970}).

For semigroups, Almeida's semigroups $Y$ and $Q$ (see \cite{Almeida1986b}), as well as Rasin's semigroups $K_p = \mathcal{M}\big(\mathbb{Z}/p\mathbb{Z}; 2, 2; \big(\begin{smallmatrix} 0 & 0 \\ 0 & 1 \end{smallmatrix}\big) \big)$ (see \cite{Rasin1979}) must be excluded for every prime~$p$.

\medskip

A short inspection of defining identities reveals that $\pv{T} \subseteq \pv{V}(N^1), \pv{V}(Y), \pv{V}(Q)$.
These identities are as follows (see \cite{Edmunds1977,Edmunds1980,Almeida1986b}):
\begin{multline*}
  \pv{V}(N^1) = \pvI{x^3 \approx x^2, x^2y \approx xyx \approx yx^2}, \quad
  \pv{V}(Y) = \pvI{x^3 \approx x^2, x^2y^2 \approx xyx \approx y^2x^2}, \\
  \pv{V}(Q) = \pvI{x^3 \approx x^2, x^2yx^2 \approx xyx, y_1^2xy_2^2 \approx y_1^2xy_2^2, xy_1xy_2x \approx xy_2xy_1x}.
\end{multline*}
In particular, the monoid $N^1$ and the semigroups $Y$ and $Q$ can be omitted from the list of excludants, as they are replaced by $\pv{T}$, if one is interested in permutative pseudovarieties rather than pseudovarieties of permutative semigroups.

Consequently, we also obtain the following characterizations for aperiodic and nilpotent semigroups,
that is, for $\pv{A} = \pvI{x^{\omega} \approx x^{\omega+1}}$ and $\pv{N} = \pvI{x^{\omega} \approx 0}$, repsectively.

\begin{corollary}
  Let $\pv{V}$ be a pseudovariety.
  Then the following statements hold.
  \begin{enumerate}
    \item If $\pv{V} \subseteq \pv{A}$, then $\pv{V}$ is permutative if and only if $\pv{T}, \pv{LRB}, \pv{RRB} \not\subseteq \pv{V}$.
    \item If $\pv{V} \subseteq \pv{N}$, then $\pv{V}$ is permutative if and only if $\pv{T} \not\subseteq\pv{V}$.
  \end{enumerate}
\end{corollary}

Let us now prepare the proof of \cref{thm:permutative}, which is essentially a combination of \cref{thm:main,thm:obstructions}, and the following two rather simple observations.

\begin{lemma}\label{lem:I-medial}
  Let $S$ be a finite semigroup, and suppose that $S$ is completely regular and permutative.
  Then $S$ is medial, i.e., it satisfies the identity $x y_1y_2 z \approx x  y_2y_1  z$.
\end{lemma}
\begin{proof}
  Indeed, $S \models x y_1y_2  z \approx x x^{\omega}  y_1y_2  z^\omega z \approx x x^{\omega}  y_2y_1  z^\omega z \approx x  y_2y_1  z$.
\end{proof}

For every finite permutative semigroup $S$, its completely regular elements form a subsemigroup $I(S) \leq S$; see~\cite[Theorem~6.5.17]{Almeida1994}.
In fact, the same conclusion can be made if $S$ satisfies any nontrivial product identity.
Since this statement may also be of independent interest, we give a short proof based on \cref{thm:main}.

\begin{lemma}\label{lem:I-subsemi}
  Let $S$ be a finite semigroup that satisfies a nontrivial product identity.
  Then its set of completely regular elements $I(S)$ is a subsemigroup of $S$.
\end{lemma}

\begin{proof}
  Suppose that $s_1, s_2 \in I(S)$, i.e., $s_1,s_2 \in S$ satisfy $s_1 = s_1^{\omega + 1}$ and $s_2 = s_2^{\omega + 1}$.
  We need to show that $s_1s_2 = (s_1s_2)^{\omega + 1}$.
  As $S$ satisfies a nontrivial product identity, it satisfies one of the identities in \cref{thm:main}.
  We distinguish between two cases.

  In the first case, we assume that $S$ satisfies the regular expansion identity
  \[
    S \models x_1 \dots x_n \approx x_1 \dots x_{i-1}(x_i \dots x_j)^{\omega + 1}x_{j+1} \dots x_n
  \]
  where, without loss of generality, $1 < i \leq j < n$.
  We then substitute the value~$s_1^\omega$ for $x_1, \dots, x_{i-1}$, the value~$s_1s_2$ for $x_i$, and the value~$s_2^\omega$ for $x_{i+1}, \dots, x_n$.
  This yields
  \[
    s_1 s_2 = s_1^\omega s_1s_2 s_2^\omega = s_1^\omega(s_1s_2)^{\omega + 1} s_2^\omega = (s_1s_2)^{\omega + 1}.
  \]

  In the second case, we assume that $S$ satisfies a permutation identity.
  It is well-known \cite[Theorem~1]{PutchaYaqub1971} (see also \cite[Exercise~6.3.9]{Almeida1994}) that this implies 
  \[
    S \models x_1 \dots x_p  y_1 y_2  z_1 \dots z_q \approx x_1 \dots x_p  y_2 y_1  z_1 \dots z_q
  \]
  for sufficiently large $p,q \geq 0$, repeated use of which yields the second equality in 
  \[
    s_1s_2 = s_1^\omega s_1 (s_1^\omega s_2^\omega) s_2 s_2^\omega = s_1^\omega s_1 (s_2s_1)^\omega s_2 s_2^\omega = (s_1s_2)^{\omega+1}. \qedhere
  \]
\end{proof}

\begin{proof}[Proof of \cref{thm:permutative}]
  The conditions $\pv{T} \not\subseteq \pv{V}$ and $\pv{V} \subseteq \pv{Perm}$ are clearly necessary.
  To see that they are also sufficient, let us first note that, by \cref{thm:main,thm:obstructions}, the condition $\pv{T} \not\subseteq \pv{V}$ implies that $\pv{V}$ satisfies a permutation identity, in which case there is nothing to prove, or the regular expansion identity
  \[
    x_1 \dots x_n \approx x_1 \dots x_{i-1}(x_i \dots x_j)^{\omega + 1}x_{j+1} \dots x_n
  \]
  for some parameters $1 \leq i \leq j \leq n$, which we will henceforth assume.
  It follows that, for some sufficiently large $k \geq 1$, $\pv{V} \subseteq \pvI{x y z \approx x y^{\omega + 1} z} \pvM \pv{N}_k$.\footnote{A finite semigroup $S$ belongs to the Mal'cev product $\pv{W} \pvM \pv{N}_k$ if and only if its ideal $S^k \leq S$ belongs to $\pv{W}$.
  Accordingly, defining identities for $\pv{W} \pvM \pv{N}_k$ may be obtained from those for $\pv{W}$ by replacing every variable $x$ with a product $x_1 \dots x_k$ of new variables \cite[Theorem~4.1]{PinWeil1996}.}

  Suppose that $\pv{V} \subseteq \pv{Perm}$.
  It then suffices to show that $\pv{V}' \coloneqq \pv{V} \cap \pvI{x y z \approx x y^{\omega + 1} z}$ is a permutative pseudovariety, since this implies the permutativity of $\pv{V} \subseteq \pv{V}' \pvM \pv{N}_k$, and we will do so by showing that $\pv{V}' \models x_1x_2y_1y_2z_1z_2 \approx x_1x_2y_2y_1z_1z_2$.
  To this end, let $S \in \pv{V}'$.
  By \cref{lem:I-subsemi}, the set $I(S)$ is a subsemigroup of $S$; hence, $I(S) \in \pv{V}'$.
  In turn, this implies that $I(S) \in \pv{Perm}$ and, by \cref{lem:I-medial}, that $I(S)$ is medial.
  Using this together with the identity $xyz \approx xy^{\omega+1}z$, we finally obtain
  \begin{align*}
    S \models x_1x_2y_1y_2z_1z_2
    &\approx x_1x_2^{\omega+1}y_1^{\omega+1} y_2^{\omega+1}z_1^{\omega+1} z_2 \\
    &\approx x_1x_2^{\omega+1}y_2^{\omega+1} y_1^{\omega+1}z_1^{\omega+1} z_2 
      \approx x_1x_2y_2y_1z_1z_2. \qedhere
  \end{align*}
\end{proof}

\section*{Acknowledgements}

The author would like to thank Markus Lohrey, Florian Stober, and Armin Weiß for valuable discussions, and gratefully acknowledges funding by the Deutsche Forschungs\-gemeinschaft (DFG, German
Research Foundation) -- LO~748/15-1.

\bibliographystyle{plainurl}
\bibliography{references}

\end{document}